\documentclass[12pt]{amsart}
 \usepackage{amsmath,amssymb,amsfonts,amsthm,amscd,textcomp,times,booktabs,mathrsfs}
   \usepackage[dvipdfmx]{graphicx}
    \usepackage{braket}
    \usepackage{tabularx}
  \usepackage{color,float}
  \usepackage{bm}
\usepackage[all]{xy}
 \usepackage{wrapfig}

\newtheorem{theorem}{Theorem}

\newtheorem{proposition}[theorem]{Proposition}
\newtheorem{lemma}[theorem]{Lemma}
\newtheorem{definition}[theorem]{Definition}
\newtheorem{corollary}[theorem]{Corollary}
\newtheorem{remark}[theorem]{Remark}
\newtheorem{example}{Example}

\numberwithin{equation}{section}
\numberwithin{theorem}{section}

\newcommand{\Z}{\ensuremath{\mathbb{Z}}}
\newcommand{\Q}{\ensuremath{\mathbb{Q}}}
\newcommand{\R}{\ensuremath{\mathbb{R}}}
\newcommand{\C}{\ensuremath{\mathbb{C}}}

\newcommand{\vol}{\ensuremath{\mathrm{vol}}}

\newcommand{\bracket}[1]{\ensuremath{\langle #1 \rangle}}

\newcommand{\restrict}[2]{\ensuremath{\left. #1 \right|_{#2}}}

\DeclareMathOperator{\Hom}{Hom}

\def\bs{\boldsymbol}

\def\D{\Delta}
\def\p{\partial}

\def\P{\mathbb{P}}

\begin{document}
\title[GIT-stability of toric Fano varieties]
{Asymptotic Chow semistability implies Ding polystability for Gorenstein toric Fano varieties}

\author{Naoto Yotsutani}

\address{Kagawa University, Faculty of education, Mathematics, Saiwaicho $1$-$1$, Takamatsu, Kagawa,
$760$-$8522$, Japan}
\email{yotsutani.naoto@kagawa-u.ac.jp}

\makeatletter
\@namedef{subjclassname@2020}{%
  \textup{2020} Mathematics Subject Classification}
\makeatother

\dedicatory{}
\date{September 22, 2023}
\maketitle

\noindent{\bfseries Abstract.}
In this paper, we prove that a {\emph{Gorenstein}} toric Fano variety $(X, -K_X)$ is asymptotically Chow semistable, then it is Ding polystable with respect to toric test configurations (Theorem \ref{thm:Main}).
This extends the known result obtained by others (Theorem \ref{thm:BOY}) to the case where $X$ admits {\emph{Gorenstein singularity}}.
We also show the additivity of the Mabuchi constant for the product toric Fano varieties in Proposition \ref{prop:ProdConst} based on the author's recent work (Ono, Sano and Yotsutani in arXiv:$2305.05924$).
Applying this formula to certain toric Fano varieties, we construct {\emph{infinitely many}} examples that clarify the difference between relative K-stability and relative Ding stability in a systematic way (Proposition \ref{prop:ProdExt}).
Finally, we verify relative Chow stability for Gorenstein toric del Pezzo surfaces using the combinatorial criterion developed in (Yotsutani and Zhou in {\emph{Tohoku Math. J.}} {\bf{71}} ($2019$), $495$-$524$.) and specifying the symmetry of the associated polytopes as well.


\section{Introduction}
Let $(X,L)$ be a polarized projective variety of complex dimension $n$.
One of the outstanding problem in K\"ahler geometry is to distinguish whether the first Chern class
$c_1(L)$ contains a K\"ahler metric $\omega$ with constant scalar curvature (cscK metric).
A parallel reasoning question in 
algebraic geometry is to study an appropriate notion of stability
of $(X,L)$ in the sense of Geometric Invariant Theory (GIT). This leads us to investigate various notions of
GIT-stability and study the relation among them. For example, Ross-Thomas clarified the following implications
among GIT-stability in their paper \cite{RT07}: 
\vskip 5pt
$\begin{matrix}
& \text{Asymptotic Chow stability} & \Rightarrow  & \text{Asymptotic Hilbert stability} \\ 
\Rightarrow  & \text{Asymptotic Hilbert semistability} & \Rightarrow  & 
\text{Asymptotic Chow semistability}  \\ 
  &  & \Rightarrow  & 
\text{K-semistability}.  \\
\end{matrix}$
\vskip 5pt

In \cite{Mab08}, Mabuchi proved that Chow stability and Hilbert stability asymptotically coincide.
Remark that for a fixed positive integer $i\in \Z_{+}$, Chow stability for $(X,L^{\otimes i})$ implies Hilbert stability for
$(X,L^{\otimes i})$ (i.e. not necessarily asymptotic stability case) by the classical result due to Forgaty \cite{Fo69}. See also \cite[Corollary $3.4$]{KSZ92} for more combinatorial description of this result 
in terms of GIT weight polytopes.

In order to describe our issue more precisely, we first recall that a complex normal variety $X$ is 
said to be {\em Fano} if its anticanonical divisor $-K_X$ is ample. It is called {\em Gorenstein} if $-K_X$ is Cartier.
Suppose that $X$ is a {\emph{smooth}} Fano variety (i.e. a Fano manifold) with a K\"ahler metric
\[
\omega = \sqrt{-1}g_{i\bar j}dz_i \wedge d\bar z_j \in 2\pi c_1(X).
\]
We recall that $\omega_{\varphi}=\omega + \sqrt{-1}\p \bar \p \varphi$ is K\"ahler-Einstein if and only if $\varphi$ is
a critical point of either the K-energy $\nu_{\omega}$ or the Ding 
functional $\mathcal D_{\omega}$ where both functionals defined on the space of K\"ahler potentials $\mathcal H_{\omega}=\set{\varphi \in C^{\infty}(X) | \omega_{\varphi} > 0}$. It is known that these functionals satisfy the inequality $\mathcal D_{\omega}\leqslant \nu_{\omega}$ which turns out to be the Ding invariant is less than or equal to the Donaldson-Futaki (DF) invariant
\cite{Ber16}. In the case where $X$ is toric, Yao gave explicit description of the inequality between the Ding
invariant and the DF invariant in terms of the associated polytope \cite[Proposition $4.6$]{YiYao17}. 
In particular, Ding polystability implies K-polystability for a toric Fano manifold. 
Moreover, the converse direction has been also proved in \cite{Fuj16} for (not necessarily toric) Fano manifolds.
\begin{theorem}[Fujita]\label{thm:BBJ-Fuj}
Let $X$ be a Fano manifold. Then Ding semistability is equivalent to K-semistability.
Furthermore, Ding polystability (resp. Ding stability) is also equivalent to K-polystability 
(resp. K-stability). 
\end{theorem}
On the one hand, in differential-geometrical point of view,
Theorem \ref{thm:BBJ-Fuj} corresponds to the fact that cscK metrics in the anticanonical classes of Fano manifolds are K\"ahler-Einstein metrics.
Recall that for a compact K\"ahler manifold $X$ with a fixed K\"ahler class $[\omega]$, 
$\varphi$ is a critical point of $\nu_{\omega}$ if and only if $\omega_{\varphi}$ is a cscK metric.
On the other hand, 
we conclude that if a Fano manifold $X$ is asymptotically Chow semistable then it is Ding {\em semistable} according to the previous argument.
In the case where $X$ is a toric Fano manifold, it is known that $X$ is K-semistable if and only if it is K-polystable \cite{Ber16, BJ20, WZ04}.
Summing up these arguments, we have the following.
\begin{theorem}[Berman, Ono, Yao]\label{thm:BOY}
Let $(X, -K_X)$ be a smooth toric Fano variety.
If $(X, -K_X)$ is asymptotically Chow semistable with respect to toric test configurations, then it is Ding polystable with respect to toric test configurations.
\end{theorem}
In this article, we show that a more general result by a combinatorial proof.
\begin{theorem}\label{thm:Main}
Let $(X, -K_X)$ be a {\emph{Gorenstein}} toric Fano variety.
If $(X, -K_X)$ is asymptotically Chow semistable with respect to toric test configurations, then it is Ding polystable with respect to toric test configurations.
\end{theorem}
Essentially, the proof of Theorem \ref{thm:Main} is based on the {\em Ehrhart reciprocity law} and the fact that any toric Fano variety is K-polystable if and only if the barycenter of the associated reflexive polytope $\D\subseteq M_\R$ is the origin.
As we mentioned above, another advantage of our combinatorial approach is that $X$ may admit Gorenstein singularity (i.e. not necessarily  smooth) in our main theorem.
However, it does not work for a $\Q$-Gorenstein toric variety since the corresponding polytope $\D$ contains not only the origin, but also other lattice points.
It also should be noted that we only assume $(X, -K_X)$ to be asymptotically Chow {\emph{semistable}} and {\em do not assume} $(X,-K_X)$ to be asymptotically Chow polystable in Theorem $\ref{thm:Main}$. 

\vskip 5pt

In the following Section $\ref{sec:ARCS}$, we discuss relative stability of toric Fano variety.
Recently, we found that there are at least four examples of smooth toric Fano variety which clarify the difference between relative K-stability and relative Ding stability in \cite{NSY23}.
In order to discover these four examples of a relatively K-polystable toric Fano variety, but which is relatively Ding unstable, 
the author focused on the geometrical description such that they are all $\P^1$-bundles over $\P^m$.
In particular, we consider the case of Picard number one projective toric varieties.  
Based on a recent argument discussed in \cite{OSY23}, we systematically construct such examples in arbitrary dimension.
\begin{proposition}[See, Corollary \ref{cor:ExtProd}]\label{prop:ProdExt}
Fixing a positive integer $r$, we consider an extremal smooth toric Fano variety $X_k$ with the associated polytope $\D_k$, for $1\leqslant k\leqslant r$.
Suppose $\theta_{\D_k}(\bs x_k)$ be the potential function of $\D_k$ defined in $\eqref{def:potential}$ with $\frac{1}{r}\leqslant \theta_{\D_k} < 1$.
For the product polytope $\D=\prod_{k=1}^r\D_k$, the associated smooth toric Fano variety $(X_\D, -K_{X_\D})$ is relatively K-polystable, but it is relatively Ding unstable.
\end{proposition}
In order to prove Proposition $\ref{prop:ProdExt}$, we shall use the following additive property of the Mabuchi constant $M_{X_\D}$ for the products of toric Fano varieties.
\begin{proposition}[See, Corollary \ref{cor:additive}]\label{prop:ProdConst}
For the product polytope $\D$ of reflexive polytopes $\D_k$ for $k=1, \dots , r$, let $M_{X_\D}$ and $M_{X_{\D_{k}}}$ be the Mabuchi constant defined in $\eqref{eq:const}$.
Then we have the equality
\[
M_{X_{\D}}=M_{X_{\D_{1}}}+\dots + M_{X_{\D_{r}}}.
\]
\end{proposition}
We give a purely combinatorial proof of Proposition \ref{prop:ProdConst} in Section $\ref{sec:Product}$.
In the following Section $\ref{sec:ChowToricdP}$, we classify Gorenstein toric del Pezzo surfaces in terms of (asymptotic) relative Chow polystability. We use the criteria $\eqref{eq:relChowWeight}$ to verify asymptotic relative
Chow stability of polarized toric variety. However, it is very difficult to verify asymptotic relative Chow stability of a given polarized toric variety because we have to prove that there exists $t_i\in \R$ satisfying the equality 
in $\eqref{eq:relChowWeight}$ for {\emph{any}} positive integer $i$. In order to solve this difficulty, we consider the {\emph{symmetry of the associated polytopes}} $\D\subset M_\R$ which works very well for two 
dimensional reflexive polygons ($16$ types). Adapting the symmetry of reflexive polygons and a combinatorial criteria $\eqref{eq:relChowWeight}$ investigated by Zhou and the author in \cite{YZ19}, we verify relative 
Chow stability of each Gorenstein toric del Pezzo surfaces.
\begin{proposition}[See, Proposition $\ref{prop:ARCS}$]
Among all $16$ isomorphism classes of Gorenstein toric del Pezzo surfaces, there are $5$ isomorphism classes of asymptotically relatively Chow polystable surfaces and $4$ isomorphism classes of asymptotically relatively Chow unstable surfaces. The remaining $7$ classes are relatively Chow polystable with respect to the anticanonical polarization.
\end{proposition}
All the results are listed in Table $\ref{table:RelChowSta}$. We also refer the reader to Table $\ref{table:Combinatorics}$ for specifying symmetry of each reflexive polygon $\D\subset M_\R$.


This paper is organized as follows. Section $\ref{sec:Preliminary}$ is a brief review of Gorenstein toric Fano
varieties, Ding stability and asymptotic Chow stability. The proof of Theorem $\ref{thm:BOY}$ is given in Section $\ref{sec:proof}$. 
Section $\ref{sec:ARCS}$ collects the results of relative algebro-geometric stability. 
In Sections $\ref{sec:Fundamental}$ and $\ref{sec:ToricRelChow}$, we recall the criteria of relative Chow stability of polarized toric variety investigated by the author and B. Zhou in \cite{YZ19}.
We prove Proposition $\ref{prop:ProdConst}$ in Section $\ref{sec:Product}$ by applying the product formulas regarding convex polytopes which was also used in \cite{OSY23}.
See Lemma $\ref{lem:Prod}$ and the proof of Proposition $\ref{prop:Prod}$ for further details.
Section $\ref{sec:ChowToricdP}$ is devoted to verify asymptotic relative Chow stability of Gorenstein toric del Pezzo surfaces.
All the results and practical values of invariants are summarized in Proposition $\ref{prop:ARCS}$ and Table $\ref{table:RelChowSta}$.


\section{Preliminary}\label{sec:Preliminary}
\subsection{Gorenstein toric Fano varieties.}\label{sec:GorToricFano}
We first recall the standard notation and basic definitions of Gorenstein toric Fano varieties,
as it can be found in \cite{CLS11}.

Let $N\cong \Z^n$ be a lattice of rank $n$, while $M=\Hom (N,\Z)$ is the $\Z$-dual of $M$. 
Let $P \subseteq N_{\R}\cong \R^n$ be a lattice polytope with ${\bf{0}}\in {\mathrm{Int}}(P)$. 
We assume that all vertices of $P$ are primitive elements in $N$. For a subset $S$ of $N_{\R}$,
we denote the positive hull of $S$ by $\mathrm{pos}(S)$ i.e. $\mathrm{pos}(S)=\sum_{v\in S}\R_{\geqslant 0}v$.
Then
\[
\Sigma_P:=\set{\mathrm{pos}(F) | F \text{ is a face of }P}
\]
forms the fan which is often called the {\emph{normal fan}} of $P$. It is well-known that the fan
$\Sigma=\Sigma_P$ associates a toric variety $X_{\Sigma}$ with the complex torus $T_N:=\mathrm{Spec}\: \C[M]$ action. Here and hereafter we denote the associated toric variety by $X$ for simplicity.
Recall that the anticanonical divisor of $X$ is given by $-K_X=\sum_{\rho}D_\rho$ where $D_\rho$ is the torus invariant Weil divisor corresponding to a ray $\rho\in \Sigma(1)$. Then the dual polytope of $P$ (w.r.t. $-K_X$)
is defined by
\[
\D=\set{y \in M_\R | \braket{x,y} \geqslant -1 \text{ for all } x\in P}
\]
which is also an $n$-dimensional (rational) polytope in $M_\R$ with ${\bf{0}}\in \mathrm{Int} (\D)$.
Then $\D$ is called {\emph{reflexive}} if it is a lattice polytope. There is a bijective correspondence
beween isomorphism classes of reflexive polytopes and isomorphism classes of
Gorenstein toric Fano varieties. 
For a fixed dimension $n$, there are only finitely many isomorphism classes of $n$-dimensional reflexive
polytopes \cite{KS98, KS00}. 
They found $1,16,4319$ and $473800776$ isomorphism classes for $n=1,2,3$ and $4$.
Throughout the paper, we assume that a (toric) Fano variety $X$ admits at worst Gorenstein singularities.

\subsection{Ding stability for Fano varieties.}
In this section, we briefly review a notion of Ding stability, see \cite{Ber16, Fuj16, YiYao17} for more details. 

\vskip 5pt

Let $(X,\omega)$ be an $n$-dimensional Fano manifold with a K\"ahler metric $\omega \in 2 \pi c_1(X)$.
We set $V$ to be the volume $\displaystyle V:=\int_X \omega^n$ of the given Fano manifold $X$. Recall that the {\em Ding functional} $\mathcal D_{\omega}: \mathcal H_{\omega} \rightarrow \R$ is given by
\[
\mathcal D_{\omega}:= -\frac{1}{V}\int_0^1\int_X \dot{\varphi}_t(1-e^{\rho_{\omega_{t}}})\omega_t^n\; dt,
\]
where $\varphi_t$ is a smooth path in $\mathcal H_{\omega}$ joining $0$ with $\varphi$ and $\rho_{\omega}$ is the
function which satisfies
\begin{equation}\label{eq:normalization}
\mathrm{Ric}(\omega)-\omega = \sqrt{-1}\p\bar \p \rho_{\omega} \quad \text{and}  \quad \int_X (e^{\rho_{\omega}}-1)\omega^n =0.
\end{equation}
Then we readily see that $\varphi$ is a critical point of $\mathcal D_{\omega}$ if and only if 
$\omega_{\varphi}$ is a K\"ahler-Einstein metric.

Next we recall a notion of a test configuration. A {\em test configuration} for a Fano variety $(X, -K_X)$ is a polarized scheme $(\mathcal X, \mathcal L)$ with:
\begin{itemize}
\item a $\C^{\times}$-action and a $\C^{\times}$-equivariant proper flat morphism $\pi:\mathcal X \rightarrow \C$, where $\C^{\times}$ acts on the base by multiplication. 
\item a $\C^{\times}$-equivariant line bundle $\mathcal L \rightarrow \mathcal X$ which is ample over all fiber
$\mathcal X_z:=\pi^{-1}(z)$ for $z\neq 0$, and $(X, -K_X)$ is isomorphic to $(\mathcal X_z, \mathcal L_z)$ with $\mathcal L_z=\mathcal L |_{\mathcal X_z}$.
\end{itemize}
Taking a Hermitian metric $h_0$ on $\mathcal O_X(-K_X)$ with positive curvature, we can construct the Phong-Sturm geodesic ray $h_t$ which emanates from $h_0$ in $\mathcal H_\omega$
\cite{PS07}. 
In \cite{Ber16}, Berman defined the Ding invariant as the asymptotic slope of the Ding functional
along the geodesic rays. Moreover he showed that
\[
DF(\mathcal X, \mathcal L)=\lim_{t \rightarrow \infty}\frac{1}{V}\frac{d \mathcal D_{\omega}(h_t)}{d t}+q
\]
where the error term $q$ is non-negative and $DF(\mathcal X, \mathcal L)$ is the {\em Donaldson-Futaki invariant}.
Then the {\em Ding invariant} $\mathrm{Ding}(\mathcal X, \mathcal L)$ is given by
\[
\mathrm{Ding}(\mathcal X, \mathcal L)=\lim_{t \rightarrow \infty}\frac{1}{V}\frac{d \mathcal D_{\omega}(h_t)}{d t}.
\]
A Gorenstein Fano variety $X$ is said to be {\em Ding-semistable} if for any test configuration $(\mathcal X, \mathcal L)$ for $(X, -K_X)$, we have $\mathrm{Ding}(\mathcal X, \mathcal L)\geqslant 0$.
Moreover $X$ is said to be {\em Ding polystable} if $X$ is Ding semistable and $\mathrm{Ding}(\mathcal X, \mathcal L)=0$ if and only if $(\mathcal X, \mathcal L)$ is equivariantly isomorphic to $(X\times \C, p_1^*(-K_X))$
where $p_1:X\times \C \rightarrow X$ is the projection.

\vskip 3pt

Now we consider the toric case. Let $X$ be an $n$-dimensional toric Fano variety and $\D \subseteq M_\R$ the corresponding reflexive polytope with the coordinates ${\boldsymbol{x}}=(x_1,\dots ,x_n)$.
Recall that a piecewise linear convex function $u=\max\set{f_1, \ldots, f_{\ell}}$ on $\D$ is called {\em rational}
if $f_k=\sum a_{k,i}x_i+c_k$ with $(a_{k,1}, \ldots, a_{k,n})\in \Q^n$ and $c_k\in \Q$ for $k=1,\ldots, \ell$.
A {\em toric test configuration} for $(X, -iK_X)$, introduced by Donaldson \cite{Do02}, is a test configuration associated with a rational piecewise linear convex function $u$ on $\D$, so that $iQ$ is a lattice
polytope in $M_\R\times \R\cong \R^{n+1}$. 
Here $Q$ is given by
\[
Q=\set{(\boldsymbol{x},t)|\boldsymbol{x}\in \D,\; 0\leqslant t \leqslant R-u(\boldsymbol{x})}
\]
and $R$ is an integer such that $u\leqslant R$.
Then $iQ$ defines the $n+1$-dimensional polarized toric variety $(\overline{\mathcal X}, \overline{\mathcal L})$
and a flat morphism $\overline{\mathcal X} \rightarrow \C P^1$.
Hence the family restricted to $\C$ gives a torus equivariant test configuration $(\mathcal X, \mathcal L)$ for
$(X, -iK_X)$. 

The toric geodesic ray $h_t$ associated to a toric test configuration was described by
Song-Zeldich \cite{SoZeld12}.
In \cite{YiYao17}, Yao detected an explicit description of the Ding invariants of toric Fano varieties.
\begin{theorem}[Yao]
Let $(X,-K_X)$ be a Gorenstein toric Fano variety with the associated reflexive polytope $\D$.
Let $u$ be a piecewise linear convex function.
The Ding invariant of the toric test configuration associated to $u$ is given by
\begin{align}
\begin{split}
\mathrm{Ding}(\mathcal X, \mathcal L)&=\lim_{t \rightarrow \infty}\frac{1}{\vol(\D)}\frac{d \mathcal D_{\omega}(h_t)}{d t} \\
& = -u(0)+\frac{1}{\vol(\D)}\int_\D u({\boldsymbol{x}})\:dv =:\mathcal I_\D (u).
\end{split}
\end{align}
\end{theorem}
Then a reflexive polytope $\D\subseteq M_\R$ is said to be {\em Ding polystable} if $\mathcal I_\D (u)\geqslant 0$
for all convex piecewise linear functions $u$ and the equality holds if and only if $u$ is affine linear.
One can observe that $\mathcal I_\D (u)$ is invariant when we add affine linear functions to convex piecewise linear
functions. Hence it suffices to consider {\em normalized} convex piecewise linear functions $u$ on $\D$ for our purpose, that is, $u({\boldsymbol{x}})\geqslant u(0)=0$.  
The following observation was given by Yao \cite{YiYao17}, and we write down the detail for the reader's convenience.
\begin{proposition}[Yao]\label{prop:YaoYi}
If $\D$ is a reflexive polytope, then the associated Gorenstein toric Fano variety $(X, -K_X )$
is Ding polystable if and only if the barycenter of $\D$ is ${\bf{0}}$.
\end{proposition}
\begin{proof}
Suppose $\D$ is Ding polystable. Hence
\begin{equation}\label{ineq:DingInv}
\frac{1}{\vol(\D)}\int_\D u({\boldsymbol{x}})\:dv\geqslant 0
\end{equation}
for any normalized convex piecewise linear function $u$.
Applying $\eqref{ineq:DingInv}$ to linear functions, i.e. $u=\pm x_i$ for $i=1, \ldots, n$ we conclude
$\int_\D {\boldsymbol{x}}\:dv=\bf 0$.

Conversely we assume that $\int_\D {\boldsymbol{x}}\:dv=\bf 0$. Then for any normalized convex piecewise linear function $u$, Jensen's inequality implies that
\[
\int_\D u({\boldsymbol{x}})\:dv \geqslant u(\int_\D {\boldsymbol{x}}\:dv) = u({\bf{0}})=0.
\]
Hence $\D$ is Ding polystable.
\end{proof}

\subsection{Asymptotic Chow stability of toric varieties}\label{sec:Chow Stability}
In this section, let us briefly recall notion of Chow stability, see \cite{Ono11, Yotsu16} for more details.

Let $X\subset \C P^N$ be an $n$-dimensional irreducible complex projective variety of degree $d\geqslant 2$.
Recall that for a projectively embedded $n$-dimensional complex subvariety $X\subset \C P^N$, the {\emph{degree}} $d$ of $X$ is a number of intersection of $X$
with a linear subspace $L$ in general position, such that $n+\dim L =N$.
Let us denote the Grassmann variety by $\mathbb G(k,\C P^N)$. We define the {\em associated hypersurface}
of $X\subset \C P^N$ by
\[
Z_X:=\set{L\in \mathbb G(N-n-1, \C P^N) | L\cap X\neq \emptyset }.
\]
Remark that the construction of $Z_X$ can be regarded as an analog of the projective dual varieties as in \cite[Chapter $1$]{GKZ94}. In fact, it is well known that $Z_X$ is an irreducible divisor in $\mathbb G(N-n-1, \C P^N)$ with $\deg Z_X=d$ in the Pl\"ucker coordinates.
Therefore there exists $R_X\in H^0(\mathbb G(N-n-1, \C P^N), \mathcal O_{\mathbb G}(d))$ such that
$Z_X=\set{R_X=0}$. We call $R_X$ the {\em $X$-resultant}.
Since there is a natural action of $\mathrm SL(N+1,\C)$ on $H^0(\mathbb G(N-n-1, \C P^N), \mathcal O_{\mathbb G}(d))$, we define GIT-stability for the $X$-resultant $R_X$ as follows.
\begin{definition}\label{def:ChowSta}\rm
Let $X\subset \C P^N$ be an $n$-dimensional irreducible complex projective variety.
$X$ is said to be {\em Chow semistable} if the closure of $\mathrm SL(N+1,\C)$-orbit of the $X$-resultant $R_X$
does not contain the origin. $X$ is said to be {\em Chow polystable} if the orbit $\mathrm SL(N+1,\C)\cdot R_X$
is closed. We call $X$ {\em Chow unstable} if it is not Chow semistable.
\end{definition}
\begin{definition}\label{def:Asymptotic Chow sta}\rm
Let $(X, L)$ be a polarized projective variety. For $i\gg 0$, we denote the Kodaira embedding by
$\Psi_i:X\rightarrow \P(H^0(X, L^{\otimes i})^*)$. $(X,  L)$ is said to be {\em asymptotically Chow semistable}
(resp. {\em polystable}) if there is an $i_0$ such that $\Psi_i(X)$ is Chow semistable (resp. polystable) for
each $i\geqslant i_0$.
$(X,  L)$ is called {\em asymptotically Chow unstable} if it is not asymptotically Chow semistable.
\end{definition}
Next we will give a quick review on Ono's necessary condition for Chow semistability of polarized toric varieties.
Let $\D$ be an $n$-dimensional lattice polytope in $M_{\R}\cong \R^n$. The {\em Euler-Maclaurin summation formula}
for polytopes provides a powerful connection between integral over a polytope $\D$ and summation of lattice points in $\D$. More specifically, for any polynomial function $\phi$ on $\R^n$, we would like to see how the summation
\[
\sum_{{\bf {a}} \in \D \cap (\Z /i)^n}\phi ({\bf{a}})=:I(\phi,\D)(i)
\]
will behave for a positive integer $i$. If we take $\phi$ to be $1$, $I(\phi,\D)(i)$ is so-called the Ehrhart polynomial which counts the number of lattice points in $i$-th dilation of a polytope $\D$:
\[
I(1, \D)(i)=\# (\D \cap (\Z/i)^n).
\]
Recall that the Ehrhart polynomial has an expression
\[
E_\D(t):=I(1,\D)(t)=\vol(\D)t^n+\frac{\vol(\p \D)}{2}t^{n-1}+\dots +1 
\]
where $\p \D$ is the boundary of a lattice polytope $\D$. Similarly if we take $\phi$ to be the coordinate functions
${\boldsymbol{x}}=(x_1,\ldots, x_n)$, then $I(\phi,\D)(i)$ counts the weight of lattice points in $i$-th dilation of a polytope $\D$:
\begin{equation}\label{eq:SumPoly}
I({\boldsymbol{x}},\D)(i)=\sum_{{\bf a}\in \D \cap (\Z/i)^n}{\bf a} .
\end{equation}
Similar to the Ehrhart polynomial, 
it is also known that $\eqref{eq:SumPoly}$ gives the $\R^n$-valued polynomial satisfying
\begin{align}
s_\D(t)&:=I({\boldsymbol{x}},\D)(t) \notag \\
&=t^n\int_\D {\boldsymbol{x}}\:dv+\frac{t^{n-1}}{2}\int_{\p \D} {\boldsymbol{x}}\:d\sigma+\dots +c, \qquad  s_\D(i)=\sum_{{\bf a}\in \D \cap (\Z/i)^n}{\bf a} \label{eq:SumPoly}
\end{align}
for any positive integer $i$.
We call $s_{\D}(t)$ the {\em lattice points sum polynomial}.
The following necessary condition of Chow semistability of projective toric varieties was obtained in \cite{Ono11}.
\begin{theorem}[Ono]
Let $\D$ be a lattice polytope, $E_\D(t)$ the Ehrhart polynomial and $s_\D(t)$ the lattice points sum polynomial.
We fix a positive integer $i\in \Z_{+}.$ If the associated toric variety $X$ with respect to $ L^{\otimes i}$ is Chow semistable,
then the following equality holds:
\begin{equation}\label{eq:Chow Weight}
s_{\D}(i)=\frac{E_\D (i)}{\vol(\D)}\int_\D {\boldsymbol{x}}\: dv.
\end{equation}
\end{theorem}
Suppose a projective polarized toric variety $(X,  L)$ associated with a lattice polytope $\D$ is {asymptotically} Chow semistable. Then there is an $i_0\in \Z_{+}$ such that $\eqref{eq:Chow Weight}$ holds 
for any positive integer $i\geqslant i_0$. On the other hand, we observe that $E_\D(t)$ and $s_\D(t)$ are
($\R^n$-valued) polynomials. Hence polynomial identity theorem gives the following (see also \cite[Theorem $1.4$]{Ono11}).
\begin{lemma}\label{lem:AsympChowss}
Let $\D$ be a lattice polytope. If the associated projective polarized toric variety $(X, L)$ is asymptotically Chow semistable, then $\eqref{eq:Chow Weight}$ holds for {\em any} (not-necessarily-positive)
integer $i\in \Z$.
\end{lemma}


\section{Proof of Theorem $\ref{thm:Main}$}\label{sec:proof}
\subsection{Ehrhart reciprocity law for polynomial functions}
Let $\D$ be an $n$-dimensional lattice polytope in $M_\R\cong \R^n$ and $\phi$ a polynomial function on $\R^n$.
As in Section \ref{sec:Chow Stability}, we consider
\[
I(\phi,\D)(i)=\sum_{{\bf{a}}\in \D \cap (\Z/i)^n}\phi({\bf{a}})
\]
and
\[
I(\phi,{\mathrm{Int}} (\D))(i)=\sum_{{\bf{a}}\in {\mathrm{Int}} (\D) \cap (\Z/i)^n}\phi({\bf{a}})
\]
for a positive integer $i$.
Remark that $I(1,{\mathrm{Int}} (\D))(i)=\# ({\mathrm{Int}} (\D) \cap (\Z/i)^n)$.
The classical result of the {\em Ehrhart reciprocity law} says that the following equality holds for any positive
integer $i\in \Z_{+}$:
\[
I(1,{\mathrm{Int}} (\D))(i)=(-1)^n I(1,\D)(-i).
\]
Brion and Vergne gave the following beautiful generalization of this reciprocity law \cite{BV97}.
\begin{theorem}[Brion-Vergne]
Let $\D$ be an $n$-dimensional lattice polytope.
If $\phi$ is a homogeneous polynomial function of degree $d$ on $\D$, then the following reciprocity law 
\begin{equation}\label{eq:BV97}
I(\phi,{\mathrm{Int}} (\D))(i)=(-1)^{n+d} I(\phi,\D)(-i)
\end{equation}
holds for any positive integer $i\in \Z_+$.
\end{theorem}  
we use this result for proving the following.
\begin{lemma}\label{lem:Reciprocity}
Let $\D$ be an $n$-dimensional reflexive polytope in $M_\R$.
Let $E_\D(t)$ be the Ehrhart polynomial, $s_\D(t)$ the lattice point sum polynomial respectively.
Then we have
\[
E_\D(-1)=(-1)^n \qquad \qquad \text{and}  \qquad \qquad s_\D(-1)={\bf{0}}.
\]
\end{lemma}
\begin{proof}
We note that ${\mathrm{Int}} (\D)\cap \Z^n=\set{\bf 0}$ since $\D$ is a reflexive polytope.
Taking $\phi=1$ and $i=1$ in $\eqref{eq:BV97}$, we have
\[
E_\D(-1)=(-1)^n\cdot \# ({\mathrm{Int}} (\D) \cap \Z^n) = (-1)^n.
\]
Similary, if we take $\phi={\boldsymbol{x}}$ and $i=1$, then $\eqref{eq:BV97}$ becomes
\[
s_\D(-1)=(-1)^{n+1}\cdot \sum_{{\bf{a}}\in {\mathrm{Int}} (\D) \cap \Z^n}{\bf{a}} = {\bf{0}}.
\]
\end{proof}

\subsection{A combinatorial proof}
Now we prove Theorem $\ref{thm:Main}$.
\begin{proof}[Proof of Theorem $\ref{thm:Main}$]
If a Gorenstein toric Fano variety $(X, -K_X)$ is asymptotically Chow semistable, then
$\eqref{eq:Chow Weight}$ holds for any integer $i\in \Z$, by Lemma $\ref{lem:AsympChowss}$.
Taking $i=-1$ in $\eqref{eq:Chow Weight}$, we have
\[
\int_\D {\boldsymbol{x}}\: dv={\bf{0}}
\]
by Lemma $\ref{lem:Reciprocity}$. Thus Proposition $\ref{prop:YaoYi}$ implies that $(X, -K_X)$
is Ding polystable. This completes the proof.
\end{proof}

\subsection{Conclusion of the proof of Theorem \ref{thm:Main}}
If $\D$ is a simple reflexive polytope, then the corresponding toric Fano variety $(X, -K_X)$
may admit only orbifold singularities. Combining Theorem $\ref{thm:Main}$ and the result of \cite{SZ12, CS15},
we conclude the following. 
\begin{corollary}
Let $X$ be a toric Fano orbifold. If $(X, -K_X)$ is asymptotic Chow semistable, then $X$ admits
a K\"ahler-Einstein metric in $c_1( -K_X)$.
\end{corollary}
We finish this section with the following example which illustrates the combinatorial proof of Theorem \ref{thm:Main} by using a Gorenstein toric del Pezzo surface. 
\begin{example}\rm
Let $\D$ be the polygon in $M_\R\cong \R^2$ whose vertices are given by
\[
\Set{\begin{pmatrix}
1 \\ 0
\end{pmatrix},
\begin{pmatrix}
0 \\ 1
\end{pmatrix},
\begin{pmatrix}
-1 \\ -1
\end{pmatrix}
},
\]
which is the polytope labeled with $9$ in Table \ref{table:Combinatorics}. Then the associated polarized toric variety $(X,L)$ is the cubic surface  
\[
X=\Set{[x:y:z:w]\in \P^3| xyz=w^3}
\]
with the anticanonical line bundle $L=\mathcal O_X(-K_X)$. It is known that $(X,L^{\otimes i})$ is Chow polystable for any integer $i>0$ by Theorem $1.2$ $(3)$ in \cite{LLSW19}.
Thus, $(X,-K_X)$ is asymptotically Chow semistable.

Let us compute the $\R^2$-valued polynomial function $s_\D(t)$. Firstly, the straight forward computation shows that
\[
\int_\D \bs x\, dv=
\begin{pmatrix}
0 \\0
\end{pmatrix}
\]
for the standard volume form $dv=dx\wedge dy$ of $M_\R$.
Secondary, we shall compute $\int_{\p\D}\bs x d\sigma$ (which is equal to the second leading coefficient of $s_\D(t)$ ) as follows:
the polygon $\D$ has three facets $F_i=\set{\bs x\in \D| \ell_i(\bs x)=0}$ for $i=1,2,3$ whose defining equations are given by
\[
\ell_1(\bs x)=1-x-y, \quad \ell_2(\bs x)=1+2x-y, \quad \text{and} \quad \ell_3(\bs x)=1-x+2y,
\]
respectively. Then the boundary measure $d\sigma_i$ on each facet $F_i$ is determined by
\begin{equation}\label{eq:BoundMes2}
dv=\pm d\sigma_i\wedge d\ell_i.
\end{equation}
Thus, $\eqref{eq:BoundMes2}$ yields that we can take
\[
d\sigma_1=-dx, \quad d\sigma_2=-dx, \quad \text{and} \quad d\sigma_3=\frac{1}{2}dx 
\]
as the boundary measures on $\p \D$. Consequently, the $x$-coordinates of the barycenter of each facet $F_i$ is given by
\begin{align*}
\int_{F_1}x\,d\sigma_1&=\int_1^0x(-dx)=\frac{1}{2},\qquad \int_{F_2}x\,d\sigma_2=\int_{-1}^0x\,dx=-\frac{1}{2}\\
\text{and}\qquad \int_{F_3}x\,d\sigma_3&=\frac{1}{2}\int_{-1}^1x\,dx=0,
\end{align*}
respectively. By the symmetry of $\D$, we find that
\[
\int_{\p \D}\bs x\,d\sigma=\begin{pmatrix}
\frac{1}{2} \\ \vspace{-0.4cm}\\ \frac{1}{2}
\end{pmatrix}+\begin{pmatrix}
-\frac{1}{2} \\ \vspace{-0.4cm}\\ 0
\end{pmatrix}+
\begin{pmatrix}
0 \\ \vspace{-0.4cm}\\ -\frac{1}{2}
\end{pmatrix}=\begin{pmatrix}
0\\ 0
\end{pmatrix}.
\]
Hence, $s_\D(t)$ has the form of
\begin{equation}\label{eq:SumPoly2}
s_\D(t)=\begin{pmatrix}
0\\ 0
\end{pmatrix}t^2+
\begin{pmatrix}
0\\ 0
\end{pmatrix}t+\begin{pmatrix}
c_1\\ c_2
\end{pmatrix}
\end{equation}
for some constants $c_1$ and $c_2$. See, $\eqref{eq:SumPoly}$.
In order to determine $c_1$ and $c_2$, we plug the value of $s_\D(1)=\begin{pmatrix}
0\\ 0
\end{pmatrix}$ into $\eqref{eq:SumPoly2}$ which yields that $c_1=c_2=0$. Thus, we see that $s_\D(t)\equiv \bs 0$ and this is consistent with the Ehrhart reciprocity low
\[
s_\D(-1)=(-1)^3\cdot \sum_{\bs a\in \mathrm{Int}(\D)\cap \Z^2}\bs a=\begin{pmatrix}
0\\ 0
\end{pmatrix}.
\]
Moreover, we already see that $\int_\D\bs x\,dv=\begin{pmatrix}
0\\ 0
\end{pmatrix}$ in the above computation. Consequently, $(X,-K_X)$ is Ding polystable by Proposition \ref{prop:YaoYi}.
\end{example}

\section{Relative algebro-geometric stability}\label{sec:ARCS}
In order to deal with the existence problem of extremal K\"ahler metrics, the definition of K-stability was extended by Sz\'ekelyhidi in \cite{Sz07} to K\"ahler classes with non-vanishing Futaki invariant 
which was called {\emph{relative K-stability}}. Analogously, we can extend the notion of Chow stability to {\emph{relative Chow stability}} which has been also investigated by many researchers \cite{Se17,Ha19}.

In this section, we study relative Chow/K-stability of toric Fano varieties which were dealt with in \cite{YZ19,NSY23}.
The product formulas for potential functions $\theta_\D$ and the additivity of the constant $M_{X_\D}$ defined in $\eqref{eq:const}$ are discussed in Section $\ref{sec:Product}$.
Then in Section $\ref{sec:ChowToricdP}$, we verify (asymptotic) relative Chow stability of Gorenstein toric del Pezzo surfaces, by applying our combinatorial criterion of relative Chow stability (see, Corollary $\ref{cor:YZ}$)  in the toric setting,
and list the results in Table $\ref{table:RelChowSta}$. In Section $\ref{sec:relDK}$, we systematically construct examples of relatively K-polystable toric Fano manifolds,
but which are relatively Ding unstable, building upon the works of \cite{NSY23} and \cite{OSY23}. 
See Corollary $\ref{cor:ExtProd}$ and Example $\ref{ex:Y23}$ for more details.

\subsection{Fundamental results on relative Chow stability}\label{sec:Fundamental}
Firstly, we quick review notion of relative Chow stability and related results. See \cite{YZ19} for more detail.

Let us consider a reductive complex algebraic group $G$ with Lie algebra $\mathfrak g$.
Suppose $G$ acts linearly on a finite dimensional complex vector space $\mathbf V$.
This induces a natural $G$-action on $\mathbb P(\mathbf V)$. We will abbreviate $v\in \mathbb P(\mathbf V)$
and its representatives in $\mathbf V$. Let $T$ be a torus in $G$ with Lie algebra $\mathfrak t$.
We assume that $T$ fixes the point $v$.
Using an inner product $\bracket{\;,\;}$ and the Lie bracket $[\;,\:]$, we define subalgebras of $\mathfrak g$ by
\begin{align*}
\mathfrak{g}_{T}&=\set{\alpha \in \mathfrak g | [\alpha, \beta]=0 \;\text{ for all }\; \beta\in \mathfrak t},\\
\mathfrak{g}_{T^{\perp}}&=\set{\alpha \in \mathfrak g_T | \bracket{\alpha, \beta}=0 \; \text{ for all } \; \beta\in \mathfrak t}.
\end{align*}
Then the corresponding  Lie group of $\mathfrak g_T$ (resp. $\mathfrak{g}_{T^{\perp}}$)
is denoted by $G_T$ (resp. $G_{T^{\perp}}$). Following classical GIT (see Section $\ref{sec:Chow Stability}$),
we call $v\in \mathbb P(\mathbf V)$ is {\emph{semistable relative to $T$}} if the closure of $G_{T^{\perp}}$ orbit
$\mathcal O_{G_{T^{\perp}}}(v)$ does not contain the origin.
$v$ is {\emph{polystable relative to $T$}} if $\mathcal O_{G_{T^{\perp}}}(v)$ is closed orbit.
$v$ is said to be {\emph{unstable relative to $T$}} if it is not semistable relative to $T$.

Let us consider relative stability of the Chow form.
For an irreducible complex projective variety $X\subset \C P^N$, we choose $G=\mathrm{SL}(N+1, \C)$ and
$T$ to be the $\C^{\times}$-action induced by extremal vector field.
\begin{definition}\rm
A complex irreducible projective variety $X\subset \C P^N$ is said to be {\emph{relatively Chow polystable
(resp. semistable, unstable)}} if the $X$-resultant $R_X$ of $X$ is $\mathrm{SL}(N+1, \C)$-polystable
(resp. semistable, unstable) relative to $T$.
\end{definition}
The definition of asymptotic relative Chow stability is analogous to Definition $\ref{def:Asymptotic Chow sta}$,
hence we do not repeat the definition in this paper (see \cite[Definition $3.6$]{YZ19}).

\subsection{Toric reduction of relative Chow stability}\label{sec:ToricRelChow}
We consider the toric case. In particular, we are interested in the case where $X$ is an $n$-dimensional Gorenstein
toric Fano variety with the associated reflexive polytope $\D \subseteq M_{\R}\cong \R^n$.
As in \cite{YiYao17}, {\em {the Ricci affine function}} $\ell_{\D}$ associated to $\D$ is the unique function
determined by $\int_{\D}\ell_{\D}u \:dv=u(0)$ for any affine linear function $u$, namely, one can solve the linear system
\[
\int_{\D}\ell_{\D}({\boldsymbol{x}}) \:dv=1, \qquad \int_{\D}\ell_{\D}({\boldsymbol{x}}) \cdot x_i \:dv=0 \qquad \text{for} \qquad
i=1,\ldots,n
\]
in order to find $\ell_{\D}({\boldsymbol{x}}) =\sum a_ix_i+c$ with $a_i$ and $c$.
Let us define the {\emph{potential function}} of $\D$ by
\begin{equation}\label{def:potential}
\theta_\D:=1-\vol(\D)\ell_\D.
\end{equation} 
Then, we consider its average
\[
\bar{\theta}_\D=\frac{1}{N+1}\sum_{j=1}^{N+1}\theta_\D(\mathbf{a}_j),
\]
where $\set{\mathbf{a}_1, \ldots ,\mathbf{a}_{N+1}}$ are lattice points in $\D$.
Denoting
\[
d_\D=(1,\ldots,1), \qquad \tilde{\theta}_\D=((\theta_\D(\mathbf{a}_1)-\bar{\theta}_\D),\ldots, (\theta_\D(\mathbf{a}_{N+1})-\bar{\theta}_\D))
\]
in $\R^{N+1}$, we can show the following.
\begin{theorem}[Theorem $3.8$ in \cite{YZ19}]\label{thm:YZ}
Let $\mathrm{Ch}(\D)$ be the Chow polytope of an $n$-dimensional Gorenstein toric Fano variety $X_{\D}\subset \C P^N$. Then $X_\D$ is relatively Chow polystable in the toric sense if and only if there exists $t\in \R$ such that
\begin{equation}\label{eq:ChowPoly}
\frac{(n+1)!\vol(\D)}{N+1}(d_\D+t\tilde{\theta}_\D)\in \mathrm{Int}(\mathrm{Ch}(\D)).
\end{equation}
\end{theorem}
Let $ \bar \theta_{i\D} =\frac{1}{E_\D(i)} \sum_{\mathbf{a}\in \D\cap (\Z/i)^n}\theta_\D(\frac{\mathbf{a}}{i})$.
Defining $d_{i\D}$ and $\tilde{\theta}_{i\D}$ by 
\[d_{i\D}(\mathbf{a})=1, \quad \tilde{\theta}_{i\D}(\mathbf{a})=\frac{\theta_\D(\mathbf{a})-\bar{\theta}_{i\D}}{i},
\quad \text{ for } \quad  \mathbf{a}\in \D\cap (\Z/i)^n, \] 
we obtain a necessary condition for the associated polarized toric variety to be asymptotically relatively Chow semistable.
\begin{corollary}[Corollary $3.11$ in \cite{YZ19}]\label{cor:YZ}
If $(X_\D, \mathcal -K_{X_\D})$ is asymptotically relatively Chow semistable,
then for any $i\in \Z_+$, there exists $t_i\in \R$ satisfying 
\begin{equation}\label{eq:relChowWeight}
\sum_{\mathbf{a}\in \D\cap (\Z/i)^n}i\mathbf{a}+t_i\sum_{\mathbf{a}\in \D\cap (\Z/i)^n}
\tilde{\theta}_{i\D}(\mathbf{a})\mathbf{a}=\frac{iE_{\D}(i)}{\vol(\D)}\int_\D \boldsymbol{x}\, dv.
\end{equation}
\end{corollary}

\subsection{Product formulas for potential functions}\label{sec:Product}
Recently, Ono, Sano and the author proved that the only Bott manifolds such that the Futaki invariant vanishes for any K\"ahler class are isomorphic to the products of the projective lines \cite{OSY23}.
The key to prove the main theorem in \cite{OSY23} is the analysis of the product of two polytopes. By applying this technique to potential functions in $\eqref{def:potential}$, we derive the product formula in this section.

Now let us discuss the {\emph{product}} of two (or more) convex polytopes. For this, we consider the full dimensional polytopes $\D_1\subseteq \R^{n_1}$ and $\D_2\subseteq \R^{n_2}$,
and define
\[
\D_1\times \D_2:=\Set{
\begin{pmatrix}
\bs x \\
\bs y
\end{pmatrix}\in \R^{n_1+n_2}| \bs x=(x_1, \ldots, x_{n_1})\in \D_1,~~ \bs y=(y_1, \ldots, y_{n_1})\in \D_2
}.
\]
Setting $\D=\D_1\times \D_2$, we see that $\D$ is a polytope of dimension $n_1+n_2(=n)$, whose any nonempty face is given by the product of a nonempty face $F$ of $\D_1$, and a nonempty 
face $G$ of $\D_2$. For $i=1,2$, let $dv_i$ be the standard volume form of $\R^{n_i}$. Then, $dv=dv_1\wedge dv_2$ defines the volume form of $\D$.

For a given arbitrary (not necessarily product) convex polytope $P$ with $\dim P=n$, we consider the functional $\mathscr{L}_P(u)$ defined by
\begin{equation}\label{def:NAKfun}
\mathscr{L}_P(u)=\int_{\partial P}u\,d\sigma -\int_P\left(\frac{\vol(\partial P)}{\vol(P)}+\theta_P \right)u\,dv.
\end{equation}
Here $u$ is a convex function, $\theta_P$ is the potential function defined in $\eqref{def:potential}$, and $d\sigma$ is the $(n-1)$-dimensional Lebesgue measure of $\p P$ defined as follows:
let $\ell_j(\bs x)=\bracket{\bs x, v_j}+c_j$ be the defining equation of a facet $F_j$ of $P$, where $c_j\in \Z$ and $v_j$ is a primitive vector.
Recall that $dv=dx_1\wedge \cdots \wedge dx_n$ is the standard volume form of $\R^n$. On each facet $F_j=\set{\bs x\in P| \ell_j(\bs x)=0}\subset \p P$, we define the $(n-1)$-dimensional Lebesque measure $d\sigma_j$ of $\p P$ by
\begin{equation}\label{eq:BoundMes}
dv=\pm d\sigma_j\wedge d\ell_j.
\end{equation}
Then $d\sigma$ is uniquely determined as the $(n-1)$-dimensional Lebesgue measure of $\p P$ so that $d\sigma_j=\restrict{d\sigma}{F_j}$, up to the sign.

Let us go back to the product polytope $\D=\D_1\times \D_2$. Let $d\sigma_1$ (resp. $d\sigma_2$) be the $(n_1-1)$-dimensional (resp. $(n_2-1)$-dimensional) Lebesgue measure of $\p \D_1$
(resp. $\p \D_2$) defined in $\eqref{eq:BoundMes}$. Since any nonempty face of $\D$ is obtained by the product of a nonempty face $F\preceq \D_1$ and a nonempty face $G\preceq \D_2$, we see that the boundary of $\D$
is written as 
\begin{equation}\label{eq:Bound_Prod}
\p \D=\p \D_1\times \D_2 \cup \D_1\times \p \D_2.
\end{equation}
Also see, $(4.18)$ in \cite{OSY23}. In particular, we find the following equalities by direct computations.
\begin{lemma}\label{lem:Prod}
Let $\D=\D_1\times \D_2$ be the product of two polytopes $\D_k$ with $\dim \D_k=n_k$ for $k=1,2$.
Let $\bs x=(x_1, \ldots , x_{n_1})$ and $\bs y=(y_1, \ldots , y_{n_2})$ be the coordinates of $\D_1$ and $\D_2$ respectively.
We denote the volume form of $\D$ (resp. $\D_k$) by $dv$ (resp. $dv_k$), and the volume form of $\p \D$ (resp. $\p \D_k$) by $d\sigma$ (resp. $d\sigma_k$).
For $i=1, \ldots, n_1$ and $j=1, \ldots, n_2$, we have
\begin{align*}
\vol(\D)&=\vol(\D_1)\vol(\D_2), \\
\int_\D x_i\, dv&=\vol(\D_2)\int_{\D_1}x_i\, dv_1, \qquad  \int_\D y_j\, dv=\vol(\D_1)\int_{\D_2}y_j\, dv_2, \\
\vol(\p \D)&=\vol(\p \D_1)\vol(\D_2)+\vol( \D_1)\vol(\p \D_2), 
\end{align*}
\begin{align*}
\int_{\p \D}x_i\, d\sigma &= \vol(\D_2)\int_{\p \D_1}x_i \, d\sigma_1+\vol(\p \D_2)\int_{\D_1}x_i \, dv_1, \qquad \text{and}\\
\int_{\p \D}y_j\, d\sigma &= \vol(\D_1)\int_{\p \D_2}y_j \, d\sigma_2+\vol(\p \D_1)\int_{\D_2}y_j \, dv_2.
\end{align*}
\end{lemma}
We finish this subsection with the following additive property of the potential functions $\theta_\D$ and the Mabuchi constants $M_{X_\D}$ for the product polytopes.
\begin{proposition}\label{prop:Prod}
Let $\D=\D_1\times \D_2$ be the product of two polytopes as in Lemma \ref{lem:Prod}.
Then the potential function $\theta_\D$ defined in $\eqref{def:potential}$ satisfies the equality
\[
\theta_\D(\bs x, \bs y)=\theta_{\D_1}(\bs x)+\theta_{\D_2}(\bs y).
\]
Moreover, for the product $\D=\prod_{k=1}^r \D_k$, we see that $\theta_\D(\bs x_1, \bs x_2, \ldots , \bs x_r)=\sum_{k=1}^r\theta_{\D_k}(\bs x_k)$.
\end{proposition}
\begin{proof}
As it was described in \cite[p.$496$]{YZ19}, the potential function $\theta_\D$ is uniquely determined by solving the $n+1$-linear system
\begin{equation}\label{condi:LS}
\mathscr{L}_{\D}(1)=0, \quad \mathscr{L}_{\D}(x_i)=0, \quad \mathscr{L}_{\D}(y_j)=0 \qquad \text{for} \quad i=1, \ldots, n_1, ~~  j=1, \ldots, n_2,
\end{equation}
where $\mathscr{L}_\D(u)$ is the function defined in $\eqref{def:NAKfun}$.
Since $\theta_{\D_k}$ is the potential function of $\D_k$ for each $k=1,2$, we have
\begin{equation}\label{eq:assump}
\mathscr{L}_{\D_1}(1)=\mathscr{L}_{\D_2}(1)=0, \quad \mathscr{L}_{\D_1}(x_i)=0, \quad \text{and} \quad \mathscr{L}_{\D_2}(y_j)=0.
\end{equation}     
In order to prove our assertion, it suffices to show that $\theta_\D(\bs x, \bs y):=\theta_{\D_1}(\bs x)+\theta_{\D_2}(\bs y)$ satisfies
the $(n+1)$-equalities in $\eqref{condi:LS}$ using our assumption $\eqref{eq:assump}$.

Firstly, we find that
\begin{align*}
\int_\D \theta_\D(\bs x, \bs y)dv&=\int_{\D_1}\theta_{\D_1}(\bs x)dv_1+\int_{\D_2}\theta_{\D_2}(\bs y)dv_2,
\end{align*}
which equals $0$, by our assumption $\mathscr{L}_{\D_1}(1)=\mathscr{L}_{\D_2}(1)=0$.

Secondly, for $i=1, \ldots, n_1$, we prove that $\mathscr{L}_{\D}(x_i)=0$. To see this, we compute that
\begin{align*}
\int_\D\left( \frac{\vol(\p \D)}{\vol(\D)}+\theta_\D(\bs x, \bs y)\right)x_i\, dv
=\frac{\vol(\p \D)}{\vol(\D)}\int_\D x_i\,dv+\int_\D\left( \theta_{\D_1}(\bs x)+\theta_{\D_2}(\bs y)\right)x_i\, dv\\
=\frac{\vol(\p \D_1)\vol(\D_2)+\vol(\D_1)\vol(\p \D_2)}{\vol(\D_1)}\int_{\D_1}x_i\, dv_1+\vol(\D_2)\int_{\D_1}\theta_{\D_1}(\bs x)x_i\, dv_1.
\end{align*}
By applying Lemma \ref{lem:Prod} into $\int_{\p \D} x_i\, d\sigma$, we find that
\[
\mathscr L_{\D}(x_i)=\vol(\D_2)\mathscr L_{\D_1}(x_i)=0,
\]
where we used $\eqref{eq:assump}$ for the last equality.

Finally, for $j=1, \ldots, n_2$, we have $\mathscr L_{\D}(y_j)=\vol(\D_1)\mathscr L_{\D_2}(y_j)=0$ in the same manner as the above computation. This completes the proof of $\theta_\D(\bs x, \bs y)=\theta_{\D_1}(\bs x)+\theta_{\D_2}(\bs y)$.

In order to see the second assertion
\[
\theta_\D(\bs x_1, \bs x_2, \ldots , \bs x_r)=\sum_{k=1}^{r}\theta_{\D_k}(\bs x_k),
\]
for the product polytope $\D=\prod_{k=1}^r\D_k$,
we use the inductive argument. Hence the assertion is verified.
\end{proof}

For later use, we consider the value of constant
\begin{equation}\label{eq:const}
M_{X_\D}=\max_{\bs x\in \D}\set{\theta_\D(\bs x)},
\end{equation}
which verifies relative Ding stability of the corresponding toric (Fano) variety.
See Section $\ref{sec:relDK}$ for further discussion.
After posting this version of the paper on arXiv (version $5$, arXiv:$1711.10113$v$5$), the author found that the following additivity of the constant $M_{X_\D}$ is mentioned by Mabuchi in \cite[Theorem $9.9$]{Mab21}
for general (not necessarily toric) Fano manifold. However, it is worth to mention that we derive a direct combinatorial proof  for the case of toric Fano manifolds from Proposition $\ref{prop:Prod}$ and $\eqref{eq:const}$.
\begin{corollary}\label{cor:additive}
Let $\D=\prod_{k=1}^r\D_k$ be the product of (reflexive) polytopes. Then the constant of $M_{X_\D}$ has the additive property such that
\begin{equation}\label{eq:additive}
M_{X_\D}=M_{X_{\D_1}}+\cdots +M_{X_{\D_r}}.
\end{equation}
\end{corollary}

\subsection{Asymptotic relative Chow stability of Gorenstein toric del Pezzo surfaces}\label{sec:ChowToricdP}
As we mentioned in Section \ref{sec:GorToricFano}, there are $16$ isomorphism classes of Gorenstein toric del Pezzo surfaces. See \cite{Ni05} for more details. 
On the one hand, relative Ding stability of Gorenstein toric del Pezzo surfaces has been verified in \cite[Example $5.14$]{YiYao17}. 
On the other hand, it is difficult to verify asymptotic relative Chow stability of polarized toric variety because we have to show there exists $t_i\in \R$ satisfying $\eqref{eq:relChowWeight}$ for {\emph{any}} positive integer $i$
(cf. \cite{LLSW19} for (not relative) Chow stability case).
However, we can solve this difficulty in the case of $2$ dimension, by using symmetry of the associated reflexive polytopes.
See Case $3$ in the proof of Proposition $\ref{prop:ARCS}$ below.
As a consequence, we verify relative Chow stability of each Gorenstein toric del Pezzo surface. We list all the results in Table $\ref{table:RelChowSta}$.

\begin{proposition}\label{prop:ARCS}
Among all $16$ isomorphism classes of Gorenstein toric del Pezzo surfaces, there are $5$ isomorphism classes of asymptotically relatively Chow polystable surfaces and $4$ isomorphism classes of asymptotically relatively Chow unstable surfaces. The remaining $7$ classes are relatively Chow polystable with respect to the anticanonical polarization ($i=1$).
\end{proposition} 

\begin{table}
\caption{Relative Chow stability of Gorenstein toric del Pezzo surfaces}\label{table:RelChowSta}
\begin{center}
\begin{tabular}{lcccc}
\toprule
Label  & & Stability   & & $t_1$ in $\eqref{eq:relChowWeight}$ \\ 
in \cite{Ni05}&     \\ \midrule
$3$ $(\C P^2)$ &  & Asymptotically relatively  Chow polystable  & & No need \\ [3pt]
$4A$  $(\C P^1\times \C P^1)$&   & Asymptotically relatively   Chow polystable  & & No need \\  [3pt]
$4B$ $(dP_8)$&  & Relatively  Chow polystable w.r.t $\mathcal O_X(-K_X)$   & & $-65/828$ \\  [3pt]
$4C$ &  & Relatively  Chow polystable w.r.t $\mathcal O_X(-K_X)$    & & $-5/72$ \\  [3pt]
$5A$ $(dP_7)$&  & Relatively  Chow polystable w.r.t $\mathcal O_X(-K_X)$   & & $69/665$   \\  [3pt]
$5B$ &  &  Asymptotically relatively   Chow unstable   & & --  \\  [3pt]
$6A$ $(dP_6)$&  & Asymptotically relatively  Chow polystable  & & No need  \\ [3pt]
$6B$ &  & Relatively  Chow polystable w.r.t $\mathcal O_X(-K_X)$ & & $-259/1944$   \\  [3pt]
$6C$ &  &  Asymptotically relatively   Chow unstable   & & --   \\  [3pt]
$6D$ &  &  Asymptotically relatively   Chow unstable   & & --   \\  [3pt]
$7A$ &  & Relatively  Chow polystable w.r.t $\mathcal O_X(-K_X)$  & & $-409/2646$  \\  [3pt]
$7B$ &  &  Asymptotically relatively   Chow unstable   & & --  \\  [3pt]
$8A$ &  & Asymptotically relatively  Chow polystable  & & No need  \\ [3pt]
$8B$ &  & Relatively  Chow polystable w.r.t $\mathcal O_X(-K_X)$  & & $-33/200$  \\  [3pt]
$8C$ &  & Relatively  Chow polystable w.r.t $\mathcal O_X(-K_X)$   & & $-3/19$   \\  [3pt]
$9$ &  & Asymptotically relatively  Chow polystable  & & No need  \\ [3pt]
\bottomrule
\end{tabular}
\end{center}
\end{table}

\begin{proof}
{\bf{Case 1}}. Note that any toric surface has at worst orbifold singularities.
There are $5$ isomorphism classes of K\"ahler-Einstein Gorenstein toric del Pezzo surfaces with the vanishing Futaki
character, that is, $\C P^2, \C P^1\times \C P^1, S_6, \C P^1\times \C P^1/\mathbb Z_2$ and $\C P^2/ \mathbb Z_3$. 
Hence relative Chow stability coincides with Chow stability
for these $5$ classes of del Pezzo surfaces.
In particular, the vanishing Futaki character i.e. $\int_{\D}{\boldsymbol{x}}\:dv={\bf{0}}$
implies $\theta_{\D}\equiv 0$. This means $\tilde{\theta}_{i\D}(\mathbf{a})={\bf{0}}$ for any $i\in \Z_+$
and a necessary condition of asymptotic relative Chow semistability of a polarized toric variety
 $\eqref{eq:relChowWeight}$ becomes
\[
\sum_{\mathbf{a}\in \D\cap (\Z/i)^n}i\mathbf{a}=\frac{iE_{\D}(i)}{\vol(\D)}\int_\D \boldsymbol{x}\, dv
\]
for all $i\in \Z_+$. Hence we obtained the same equality in $\eqref{eq:Chow Weight}$.
Moreover, $\int_{\D}{\boldsymbol{x}}\:dv={\bf{0}}$ implies that $\sum_{\mathbf{a}\in \D\cap (\Z/i)^n}i\mathbf{a}={\mathbf{0}}$ for any $i\in \Z_+$. 
Remark that this is equivalent to the vanishing of the obstruction for asymptotic Chow semistability defined in \cite{Mab04} (see, \cite[p.$1385$]{Ono11}).
Since $X$ admits a K\"ahler-Einstein metric, it must be 
asymptotically Chow polystable for $X=\C P^2$, $\C P^1\times \C P^1$ and $S_6$ due to the result in \cite[Main Theorem]{Mab05}. 
Hence we verified the assertion for these $3$ classes.   

For the remaining two orbifolds cases $X=\C P^2/\Z_3$ (labeled $9$ in Table $\ref{table:RelChowSta}$) and $\C P^1\times \C P^1/\Z_2$ (labeled $8A$ in Table $\ref{table:RelChowSta}$), asymptotic
Chow polystability of $(X,-K_X)$ has been verified in Theorem $1.2$ $(3)$ in \cite{LLSW19}. We remark that the minimal embeddings of these del Pezzo surfaces are iven by
\[
\C P^2/\Z_3=\set{[z_0:z_1:z_2:z_3]\in \C P^3 | z_0^3-z_1z_2z_3=0}
\]
with three $A_2$ singularities, and
\[
\C P^1\times \C P^1/\Z_2=\set{[z_0: z_1: z_2 : z_3 :z_4]\in \C P^4 | z_1z_3-z_0^2=0, ~~ z_2z_4-z_0^2=0}
\]
with four $A_1$ singularities, respectively. See \cite{KN09} for further details.

\vskip 5pt

{\bf{Case 2}}. Let $X$ be a Goresntein toric del Pezzo surface labeled with $5B$ in Table $\ref{table:RelChowSta}$.
Then the associated reflexive polytope $\D\subseteq M_{\R}$ is given by
\[
\D=\text{conv}\set{(-1,0), (1,-2), (0,1), (-1,1)}.
\]
We claim that $X$ is asymptotically relatively Chow unstable by using Corollary $\ref{cor:YZ}$.
Hence it suffices to show that there is no $t_1\in \R$ satisfying $\eqref{eq:relChowWeight}$ for $i=1$.
See Remark $3.12$ and Proposition $5.4$ in \cite{YZ19}.
We readily see that
\begin{align*}
E_\D (i)&=\frac{5}{2}i^2+\frac{5}{2}i+1, \;  \displaystyle \int_\D \boldsymbol{x}\, dv=\left(-\frac{1}{3}, -\frac{1}{3} \right), \;
 \sum_{\mathbf{a}\in \D \cap \Z^2}\mathbf{a} =(-1,-1),\\ 
\textstyle \theta_\D(\boldsymbol{x})&=-\frac{1}{529}(1032x_1+648x_2+224) \quad \text{ and } \quad \bar{\theta}_\D=\frac{56}{529}.
\end{align*}
Therefore
\[
t_1\sum_{\mathbf{a}\in \D \cap \Z^2}\tilde{\theta}_\D(\mathbf{a})\mathbf{a} =-t_1\left(\frac{872}{529}, \frac{1160}{529} \right).
\]
This yields that there is no $t_1\in \R$ satisfying $\eqref{eq:relChowWeight}$.

\vskip 5pt

{\bf{Case 3}}. Let $X$ be a weighted projective space $\C P(1,1,2)$. This is a Gorenstein toric del Pezzo surface labeled with
$8C$ in Table $\ref{table:RelChowSta}$ and the corresponding reflexive polytope $\D$ is
\[
\D=\text{conv}\set{(-1,2), (1,0), (-1,-2)}.
\]
We prove that $(X, -K_X)$ is relatively Chow polystable.
A straightforward computation shows that
\begin{align*}
E_\D (i)&=4i^2+4i+1, \;  \displaystyle \int_\D \boldsymbol{x}\, dv=\left(-\frac{4}{3}, 0 \right), \;
 \sum_{\mathbf{a}\in \D \cap \Z^2}\mathbf{a} =(-4,0),\\ 
\textstyle \theta_\D(\boldsymbol{x})&=-\frac{3}{2}x_1-\frac{1}{2} \quad \text{ and } \quad \bar{\theta}_\D=\frac{1}{6}.
\end{align*}
Taking $i=1$ in $\eqref{eq:relChowWeight}$, we find that $t_1=-3/19$ satisfies the equation
\[
\sum_{\mathbf{a}\in \D\cap \Z^2}\mathbf{a}+t_1\sum_{\mathbf{a}\in \D\cap \Z^2}
\tilde{\theta}_{\D}(\mathbf{a})\mathbf{a}=\frac{E_{\D}(1)}{\vol(\D)}\int_\D \boldsymbol{x}\, dv.
\]
Moreover $\D$ is invariant under unimodular transformation $\begin{pmatrix} 1 & 0 \\ 0 & -1 \end{pmatrix}$
which gives the coordinate interchange $x_2 \mapsto -x_2$.
By this symmetry we conclude that there exists $t_i$ for {\emph{any}} $i\in \Z_+$ such that $\eqref{eq:relChowWeight}$ holds.

Next we verify $\eqref{eq:ChowPoly}$. For $t=-3/19$, we readily see that the left hand side of $\eqref{eq:ChowPoly}$
is given by $p:=\frac{4}{19}(11,14,17,14,11,11,11,11,14)$.
On the other hand, the Chow polytope $\mathrm{Ch}(\D)$ is the $6$-dimensional polytope in $\R^9$ with $296$ vertices.
In particular, $\mathrm{Ch}(\D)$
\footnote{We used package {\tt{TOPCOM}} for the computation.} 
is determined by $3$ defining equations $f_i(\boldsymbol{x})=0 \; (i=1,2,3)$ and $26$ 
defining inequalities $h_j(\boldsymbol{x})\geqslant 0 \; (j=1 , \ldots, 26)$ in $\R^9$.
By direct computation, one can see that $f_i(p)=0$ and $h_j(p)>0$ hold for all $i,j$. This implies
$p\in \mathrm{Int}(\mathrm{Ch}(\D))$ and the assertion is verified. Other cases are similar and further details are left to the reader.
\end{proof}
\begin{remark}\rm
\begin{enumerate}
\item Using symmetry of polytopes, one can verify the existence of $t_i$ for $i\gg 0$ satisfying $\eqref{eq:relChowWeight}$ of each case ($4B, 4C, 5A, 6B, 7A, 8B$ and $8C$ in Table $\ref{table:RelChowSta}$).
We mention that this is only a necessary condition for $(X, L)$ to be asymptotically relatively Chow semistable (Corollary $\ref{cor:YZ}$).\\
\item On the other hand, $\mathrm{Ch}(i\D)$ will be a huge number of vertices in a multidimensional Euclidean space if $i>0$ is a sufficiently large positive integer.
Hence, it is generally impossible to verify the condition
\[
\frac{i^n(n+1)!\: \vol(\D)}{E_\D(i)}(d_{i\D}+t_i\tilde{\theta}_{i\D})\in \mathrm{Int}(\mathrm{Ch}(i\D))
\]
for arbitrary positive integer $i$.
See \cite{KSZ92} and \cite{GKZ94} for more combinatorial descriptions of $\mathrm{Ch}(\D)$.
\end{enumerate}
\end{remark}

\subsection{Relative Ding/K-stability}\label{sec:relDK}
In \cite{NSY23}, we found that there are several examples of toric Fano manifolds which clarify the difference between relative K-stability and relative Ding stability.
More specifically, we verified that if $X$ is one of
\begin{itemize}
\item a toric Fano $3$-fold $\mathcal{B}_1=\P_{\P^2}(\mathcal O\oplus\mathcal{O}(2))$, or
\item toric Fano $4$-folds (which are all $\P^1$-bundles over $\P^3$) $B_1=\P_{\P^3}(\mathcal{O}\oplus \mathcal{O}(3))$, $B_2=\P_{\P^3}(\mathcal{O}\oplus \mathcal{O}(2))$, $L_1=\P_{\P^3}(\mathcal{O}\oplus \mathcal{O}(1,1,1))$,
\end{itemize}
then $(X,-K_X)$ is relatively K-polystable, but it is relatively Ding unstable. In order to prove that these four examples ($\mathcal{B}_1$, $B_1$, $B_2$ and $L_1$) admit extremal K\"ahler metrics in their first 
Chern classes, which in turn to be relatively K-polystable, we focused on their geometric structure such as projective bundles, Bott structures, etc \cite{ACGT-F08, BCT-F19, Gu95, Hw94}. 
On the one hand, relative Ding stability of toric Fano manifolds is determined by the value of constant $M_{X_\D}$ defined in $\eqref{eq:const}$
is larger than $1$ or not, due to the work of Yao \cite{YiYao17}. On the other hand, Proposition $\ref{prop:Prod}$ implies that the products of (higher dimensional) toric extremal manifolds are more likely to be
relatively Ding unstable, by the additive property of $M_{X_\D}$ (see, \eqref{eq:additive} and Corollary $\ref{cor:ExtProd}$).
In this section, we systematically construct examples of a relatively K-polystable toric Fano manifold, but it is relatively Ding unstable. 

Let us quick review the notion of relative K-stability and relative Ding stability for (smooth) toric Fano variety.
Remark that we only consider {\emph{toric}} (or {\emph{$T$-equivariant}}) {\emph{test configuration}} for the definitions of relative Ding/K-stability.
This is because for polarized toric varieties, it suffices to check only toric test configurations of relative Ding/K-stability as in \cite{De20} and \cite{LL22}. 
We refer the reader to Section $2$ in \cite{NSY23}, for more details.

Let $\D\subseteq M_\R$ be an $n$-dimensional reflexive Delzant polytope. In this case, the average of the scalar curvature, i.e., $\overline{S}={\vol(\p \D)}/{\vol(\D)}$ equals to $n$, and hence
the functional defined in $\eqref{def:NAKfun}$ will be
\[
\mathscr{L}_\D(u)=\int_{\p \D}u\,d\sigma -\int_\D(n+\theta_\D)u\,dv,
\]
where $u$ is a convex function of $\D$. A convex function $u:\D\to \R$ is called {\emph{rational PL convex}} if $u$ has the form of
\[
u({\bs{x}})=\max\set{f_1(\bs x), \ldots, f_m(\bs x)}
\]
with each $f_k$ a rational affine function. The associated anticanonically polarized smooth toric Fano variety $(X_\D, -K_{X_\D})$ is {\emph{relatively K-polystable}} if 
$\mathscr{L}_\D(u)\geqslant 0$ for any rational PL convex function $u$, and the equality holds if and only if $u$ is affine linear.
Let $M_{X_\D}$ be the {\emph{Mabuchi constant}} defined in $\eqref{eq:const}$. $(X_\D,-K_{X_\D})$ is {\emph{relatively Ding polystable}} if $M_{X_\D}\leqslant 1$.
Conversely, it is called {\emph{relatively Ding unstable}} if $M_{X_\D}> 1$. See, \cite{YiYao17} and \cite[Proposition $1.2$]{NSY23} for further details.

On the other hand, Corollary $\ref{cor:additive}$ implies that $X_\D$ is more likely to be relatively Ding unstable if the dimension of $X_\D$ is getting higher and higher.
Meanwhile, for given extremal K\"ahler manifolds $(X_k, g_k)$ with $1\leqslant k \leqslant r$, the product manifold $X=\prod_{k=1}^r X_k$ admits the product extremal K\"ahler metric $\prod_{k=1}^rg_k$.
Thus, $(X,-K_X)$ must be relatively K-polystable. In particular, $X$ is Fano. From this observation, one can expect that there are more examples of toric Fano manifolds which clarify
the difference between relative K-stability and relative Ding stability. As a consequence of $\eqref{eq:additive}$, we systematically construct infinitely many examples of relatively K-polystable
extremal toric Fano manifolds which are relatively Ding unstable.
\begin{corollary}\label{cor:ExtProd}
For $1\leqslant k \leqslant r$, let $X_k$ be an extremal toric Fano manifold with the associated polytope $\D_k$ and let $\theta_{\D_k}(\bs x_k)$ be the potential function of $\D_k$ satisfying
$\frac{1}{r}\leqslant \theta_{\D_k} < 1$. 
Let $\D$ be the product of polytopes $\D_k$ for $1\leqslant k \leqslant r$. Then the associated anticanonically polarized toric Fano manifold $(X_\D,-K_{X_\D})$ is relatively K-polystable, but it is relatively Ding unstable.
\end{corollary}
Using Table 3 in \cite{NSY23}, we obtain the following examples.
\begin{example}\label{ex:Y23}\rm
Let $dP_{9-i}$ denote a smooth del Pezzo surface with degree $(9-i)$ which is obtained by the blow-up of $\P^2$ at $i$ points. Fixing a positive integer $r$, we denote a copy of $dP_8$ by $X_k$
for $1\leqslant k\leqslant r$. It is known that $X_k$ admits an extremal K\"ahler metric in every K\"ahler class \cite{Ca82}, and this yields that $X=\prod_{k=1}^rX_k$ also admits the extremal 
K\"ahler metric in its first Chern class. Hence $(X,-K_X)$ is relatively K-polystable for any positive integer $r$.

On the other hand, the direct computation shows that $M_{X_k}=5/11$. See, \cite[Table1, No.3]{NSY23}.
Thus, we conclude that $M_X=5r/11$ by $\eqref{eq:additive}$. Consequently, $(X,-K_X)$ is relatively (uniform) Ding polystable if $r=2$, whereas it is relatively Ding unstable if $r\geqslant 3$.
We note that the toric Fano $4$-fold $dP_8\times dP_8$ is denoted by $L_7$ (No. $55$) in \cite[Table 3]{NSY23}.
In particular, there are other examples such as $Q_{10}=dP_7\times dP_8$ (No. $93$) and $dP_7\times dP_7$ (No. $119$) in four dimensional case.
\end{example}

 
\noindent {\bf{Acknowledgements.}} It is my pleasure to thank Yi Yao for his helpful e-mail and valuable comments.
This work was partially supported by JSPS KAKENHI Grant Number JP$18$K$13406$, JP$22$K$03316$, and Kagawa University Research Promotion Program $2021$ (KURPP).

\bigskip

\noindent {\bf{Data Availability Statement.}} 
Data sharing not applicable to this article as no datasets were generated or analysed during the current study.

\bigskip

\noindent{\bf{Declaration.}}
The author has no competing interests to declare that are relevant to the content of this article.
\newpage

\begin{table}[H]
\caption{Combinatorial data and the delta invariant of Gorenstein toric del Pezzo surfaces}\label{table:Combinatorics}
\begin{center}
\begin{tabular}{clc}
\toprule
Label   
&~~  $\D\subseteq M_{\R}$ & Symmetry of $\D$  \\ 
in \cite{Ni05} 
&  & \\ \midrule
$3$  
& $\mathrm{conv}\Set{\begin{pmatrix} -1\\ 1 \\  \end{pmatrix}, \begin{pmatrix}2\\ 1 \\  \end{pmatrix}, 
\begin{pmatrix}-1\\ -2 \\  \end{pmatrix}}$ & \text{No need} \\ [15pt]
$4A$   
& $\mathrm{conv}\Set{\begin{pmatrix}-1\\ 1\\  \end{pmatrix}, \begin{pmatrix}-1\\ -1 \\  \end{pmatrix}, 
\begin{pmatrix}1\\ -1 \\  \end{pmatrix},\begin{pmatrix} 1\\ 1 \\  \end{pmatrix}}$ & \text{No need} \\  [15pt]
$4B$   
& $\mathrm{conv}\Set{\begin{pmatrix}0\\ 1 \\  \end{pmatrix}, \begin{pmatrix}1\\ 1 \\  \end{pmatrix}, 
\begin{pmatrix}1\\ 0 \\  \end{pmatrix},\begin{pmatrix} -1 \\ -1 \\  \end{pmatrix}}$ & 
$\begin{pmatrix} 0 & 1 \\ 1 & 0 \end{pmatrix}
$\\  [15pt]
$4C$    
& $\mathrm{conv}\Set{\begin{pmatrix}-1\\ 1 \\  \end{pmatrix}, \begin{pmatrix}1\\ 0 \\  \end{pmatrix}, 
\begin{pmatrix}-1\\ -1 \\  \end{pmatrix}}$ & 
$\begin{pmatrix} 1 & 0 \\ 0 & -1 \end{pmatrix}
$\\  [15pt]
$5A$   
& $\mathrm{conv}\Set{\begin{pmatrix}-1\\ -1 \\  \end{pmatrix}, \begin{pmatrix}0\\ -1 \\  \end{pmatrix}, 
\begin{pmatrix}1\\ 0 \\  \end{pmatrix},\begin{pmatrix} 0 \\ 1 \\  \end{pmatrix}, \begin{pmatrix} -1 \\ 0 \\  \end{pmatrix}}$ & 
$\begin{pmatrix} 0 & 1 \\ 1 & 0 \end{pmatrix}
$\\  [15pt]
$5B$   
& $\mathrm{conv}\Set{\begin{pmatrix}-1\\ 0 \\  \end{pmatrix}, \begin{pmatrix}1\\ -2 \\  \end{pmatrix}, 
\begin{pmatrix}0\\ 1 \\  \end{pmatrix},\begin{pmatrix} -1 \\ 1 \\  \end{pmatrix}}$ & 
-- \\  [15pt]
$6A$ 
& $\mathrm{conv}\Set{\begin{pmatrix}1\\ 0 \\  \end{pmatrix}, \begin{pmatrix}0\\ 1 \\  \end{pmatrix}, 
\begin{pmatrix}-1\\ 1 \\  \end{pmatrix}, \begin{pmatrix}-1\\ 0 \\  \end{pmatrix},
\begin{pmatrix} 0 \\ -1 \end{pmatrix}, \begin{pmatrix} 1 \\ -1 \\  \end{pmatrix}}$ & \text{No need} \\  [15pt]
$6B$   
& $\mathrm{conv}\Set{\begin{pmatrix}-1\\ 1 \\  \end{pmatrix}, \begin{pmatrix}0\\ 1 \\  \end{pmatrix}, 
\begin{pmatrix}1\\ 0 \\  \end{pmatrix},\begin{pmatrix} 0 \\ -1 \\  \end{pmatrix}, \begin{pmatrix} -1 \\ -1 \\  \end{pmatrix}}$ & 
$\begin{pmatrix} 1 & 0 \\ 0 & -1 \end{pmatrix}
$\\  [15pt]
$6C$   
& $\mathrm{conv}\Set{\begin{pmatrix}-1\\ -1 \\  \end{pmatrix}, \begin{pmatrix}0\\ -1 \\  \end{pmatrix}, 
\begin{pmatrix}1\\ 1 \\  \end{pmatrix},\begin{pmatrix} -1 \\ 1 \\  \end{pmatrix}}$ & 
-- \\  [15pt]
$6D$    
& $\mathrm{conv}\Set{\begin{pmatrix}-1\\ 1 \\  \end{pmatrix}, \begin{pmatrix}-1\\ -2 \\  \end{pmatrix}, 
\begin{pmatrix}1\\ 1 \\  \end{pmatrix}}$ & 
-- \\  [15pt]
$7A$   
& $\mathrm{conv}\Set{\begin{pmatrix}0\\ 1 \\  \end{pmatrix}, \begin{pmatrix}1\\ 0 \\  \end{pmatrix}, 
\begin{pmatrix}1\\ -1 \\  \end{pmatrix},\begin{pmatrix} -1 \\ -1 \\  \end{pmatrix}, \begin{pmatrix} -1 \\ 1 \\  \end{pmatrix}}$ & 
$\begin{pmatrix} 0 & 1 \\ 1 & 0 \end{pmatrix}
$\\  [15pt]
$7B$   
& $\mathrm{conv}\Set{\begin{pmatrix}1\\ 0 \\  \end{pmatrix}, \begin{pmatrix}1\\ 1 \\  \end{pmatrix}, 
\begin{pmatrix}-3\\ -1 \\  \end{pmatrix},\begin{pmatrix} 0 \\ -1 \\  \end{pmatrix}}$ & 
-- \\  [15pt]
$8A$   
& $\mathrm{conv}\Set{\begin{pmatrix}1\\ 0 \\  \end{pmatrix}, \begin{pmatrix}0\\ 1 \\  \end{pmatrix}, 
\begin{pmatrix}-1\\ 0 \\  \end{pmatrix},\begin{pmatrix}0 \\-1 \\  \end{pmatrix}}$ & \text{No need} \\  [15pt]
$8B$    
& $\mathrm{conv}\Set{\begin{pmatrix}-1\\ 2 \\  \end{pmatrix}, \begin{pmatrix}2\\ -1 \\  \end{pmatrix}, 
\begin{pmatrix}0\\ -1 \\  \end{pmatrix},\begin{pmatrix} -1 \\ 0 \\  \end{pmatrix}}$ & 
$\begin{pmatrix} 0 & 1 \\ 1 & 0 \end{pmatrix}
$\\  [15pt]
$8C$   
& $\mathrm{conv}\Set{\begin{pmatrix}-1\\ 2 \\  \end{pmatrix}, \begin{pmatrix}1\\ 0 \\  \end{pmatrix}, 
\begin{pmatrix}-1\\ -2 \\  \end{pmatrix}}$ & 
$\begin{pmatrix} 1 & 0 \\ 0 & -1 \end{pmatrix}
$\\  [15pt]
$9$   
& $\mathrm{conv}\Set{\begin{pmatrix}1\\ 0 \\  \end{pmatrix}, \begin{pmatrix}0\\ 1 \\  \end{pmatrix}, 
\begin{pmatrix}-1\\ -1\\  \end{pmatrix}}$ & \text{No need} \\ [15pt]
\bottomrule
\end{tabular}
\end{center}
\end{table}


\end{document}